\documentclass{amsart}

\usepackage{amssymb}
\usepackage{graphicx}
\usepackage{amsmath}
\usepackage{amsthm}
\usepackage{pgf,tikz}
\usepackage{mathrsfs}
\usetikzlibrary{arrows}
\usepackage[all]{xy}
\usepackage[unicode=true,
 bookmarks=false,
 breaklinks=false,pdfborder={0 0 1},backref=section,colorlinks=false]
 {hyperref}

\newtheorem{theorem}{Theorem}

\newtheorem{corollary}[theorem]{Corollary}
\newtheorem{definition}[theorem]{Definition}

\newtheorem{lemma}[theorem]{Lemma}

\newtheorem{proposition}[theorem]{Proposition}
\newtheorem{remark}[theorem]{Remark}

\begin{document}

\title[An alternative definition of tense operators]
{An alternative definition of tense operators on residuated lattices}

\address{CIC and NUCOMPA, Facultad de Ciencias Exactas, Universidad Nacional del Centro, Pinto 399, 7000 Tandil, Argentina.}
\email{ismaelcalomino@gmail.com}
\author{Ismael Calomino}

\address{CONICET and Instituto de Ciencias B\'{a}sicas, Universidad Nacional de San Juan, 5400 San Juan, Argentina.}
\email{gpelaitay@gmail.com}
\author{Gustavo Pelaitay}

\address{CONICET and Departamento de Matem\'{a}tica, Facultad de Ciencias Exactas, Universidad Nacional del Centro, Pinto 399, 7000 Tandil, Argentina.}
\email{wizubo@gmail.com}
\author{William Zuluaga Botero}

\subjclass[2010]{Primary 03B44; Secondary 06D20, 06D50}
\keywords{Residuated lattice, tense operator, Kalman's construction}
\thanks{This work was supported by Consejo Nacional de Investigaciones Cient\'{i}ficas y T\'{e}cnicas (PIP 11220170100195CO and PIP 11220200100912CO, CONICET-Argentina) and Agencia Nacional de Promoci\'{o}n Cient\'{i}fica y Tecnol\'{o}gica (PICT2019-2019-00882, ANPCyT-Argentina). This project has also received funding from MOSAIC Project 101007627 (European Union's Horizon 2020 research and innovation programme under the Marie Sklodowska-Curie).}

\begin{abstract}In this paper we introduce and study an alternative definition of tense operators on residuated lattices. We give a categorical equivalence for the class of tense residuated lattices, which is motivated by an old construction due to J. Kalman. The paper concludes with some applications regarding the description of congruences and a 2-contextual translation. 
\end{abstract}

\maketitle

\section{Introduction}

Classical tense logic is a logical system obtained from introducing within the classical propositional logic the notion of time, i.e., an expansion of propositional logic by new unary operators which are called tense operators. It is customary to denote these  operators by $G$, $H$, $F$ and $P$, and usually we define $F$ and $P$ via $G$ and $H$ as $F(x) = \neg G(\neg x)$ and $P(x) = \neg H(\neg x)$, where $\neg x$ denote the Boolean negation of $x$. It is well known that the class of tense Boolean algebras provides an algebraic semantics for classical tense logic \cite{Burgess}. A {\it{tense Boolean algebra}} is a structure $\langle B, G, H \rangle$ such that $B$ is a Boolean algebra and $G$ and $H$ are unary operators on $B$ satisfying the following conditions:
\begin{enumerate}
\item[(TB1)] $G(1)=1$ and $H(1)=1$,
\item[(TB2)] $G(x \wedge y) = G(x) \wedge G(y)$ and $H(x \wedge y) = H(x) \wedge H(y)$,
\item[(TB3)] $x \leq GP(x)$ and $x \leq HF(x)$. 
\end{enumerate}
Note that from conditions (TB1) and (TB2) it is the case that both operators $G$ and $H$ are in fact, certain kind of modal operators, so the class of tense Boolean algebras is, in particular, a class of Boolean algebras with operators. Such class of alebras was studied in detail by Jónsson and Tarski in \cite{JonTar}.

Tense operators in intuitionistic logic were introduced by Ewald in \cite{Ewald}. In this paper it was established the corresponding intuitionistic tense logical system, called IKt. The axiomatization of Ewald is not minimal and in contrast to classical tense logic, the tense operators $F$ and $P$ cannot be defined in terms of $G$ and $H$. In \cite{FigPel14}, Figallo and Pelaitay gave an algebraic axiomatization of the IKt system and showed that the algebraic axiomatization given by Chajda in \cite{Chajda11} of the tense operators $P$ and $F$ in intuitionistic logic is not in accordance with the Halmos approach to existential operators \cite{Halmos}. Thereafter, the study of tense operators has been extended to different algebraic structures associated with non-classical logics \cite{Chiri,DiaGeor,ChajdaPaseka}. This can be evidenced by the approach applied in  \cite{FigPel14,CasCeSan}, and also in \cite{PelZu23}, in where a Kalman's construction for the class of tense distributive lattices with implication was studied.

Following the methodology proposed in \cite{Halmos,Chajda11,ChajdaPasekabook, Bak} and \cite{Paad}, the research on algebraic properties of tense operators have been taken into the context of residuated lattices. Motivated by the results presented in \cite{FigPel14,PelZu23}, the aim of this paper is to propose an alternative definition of tense operators on residuated lattices different from the one presented in \cite{Bak}. We also establish a categorical equivalence through an adaptation of the Kalman's construction developed in \cite{PelZu23} and moreover, we present some applications. 

The paper is organized as follows. In Section \ref{sec2} we review some results on integral commutative residuated distributive lattices, or ICRDL-algebras for short. In Section \ref{sec3} we present a definition of tense operators on ICRDL-algebras, unalike from the one  depicted in \cite{Bak}, and we prove a Glivenko style theorem for ICRDL-algebras as a generalization of similar results for Heyting algebras. In Section \ref{sec4}, and following the results given in \cite{CMS,CLS}, we introduce the class of tense c-differential residuated lattices, or tense DRL-algebras for short. In Section \ref{sec5}, we extend the Kalman's construction to establish a categorical equivalence between the category of tense ICRDL-algebras and the category of tense DRL-algebras. In Section \ref{sec6}, we prove that there is an isomorphism between the congruence lattice of a tense ICRDL-algebra $\mathbf{L}$ and the congruence lattice of the tense ICRDL-algebra arising from the Kalman functor $K(\mathbf{L})$. Furthermore, as an application, we see that there is a correspondence between tense filters of the classes of tense ICRDL-algebras and tense DRL-algebras. Finally, in Section \ref{sec7} we apply the results given in Section \ref{sec6} to obtain a nontrivial finite 2-contextual translation between the equational consequence relations relative to tense DRL-algebras and tense ICRDL-algebras, respectively. The reader is assumed to be familiar with standard results about adjoint functors as presented in \cite{McL1978}.

\section{Preliminaries} \label{sec2}

In this section we recall some results that will be useful through this paper. For more details on residuated lattices the reader can see \cite{GaJiKoOno}.

\begin{definition} \cite{GaJiKoOno} 
An integral commutative residuated distributive lattice, or ICRDL-algebra for short, is a structure $\langle L, \vee, \wedge,\cdot, \to, 0, 1 \rangle$ of type $(2,2,2,2,0,0)$such that:
\begin{itemize}
\item [(R1)] $\langle L, \vee, \wedge, 0, 1 \rangle$ is a bounded distributive lattice,
\item [(R2)] $\langle L, \cdot, 1 \rangle$ is a commutative monoid, 
\item [(R3)] $x \cdot y \leq z$ if and only if $x \leq y \to z$, for all $x,y,z \in L$.
\end{itemize}
As usual, in what follows, we will denote ICRDL-algebras by their underlying set $L$, when it is clear from the context.
\end{definition}

\begin{proposition} \cite{GaJiKoOno}
Let $L$ be a ICRDL-algebra. Then, the following hold:
\begin{enumerate}
\item $1 \to x = x$,
\item $x \cdot y \leq x, y$,
\item $y \leq x \to y$,
\item $x \leq y$ if and only if $x \to y=1$,
\item $x \to y \leq (x \cdot z) \to (y \cdot z)$,
\item $x \leq y$ implies $x \cdot z \leq y \cdot z$,
\item $x \leq y$ implies $z \to x \leq z \to y$ and $y \to z \leq x \to z$,
\item $x \to (y \to z) = (x \cdot y) \to z$,
\item $(x\vee y) \rightarrow z = (x \rightarrow z) \wedge (y\rightarrow z)$.
\end{enumerate}
\end{proposition}

We recall that in every ICRDL-algebra, we can define a unary operation $\neg$ as $\neg x := x \to 0$. The subsequent proposition summarizes the properties of this operation.

\begin{proposition}  \label{prop_neg}
Let $L$ be a ICRDL-algebra. Then, the following hold:
\begin{enumerate}
\item $x \leq  y$ implies $\neg y \leq \neg x$,
\item $x \leq \neg \neg x$,
\item $\neg \neg \neg x = \neg x$,
\item $\neg 1 = 0$ and $\neg 0 =1$,
\item $x \cdot \neg x = 0$,
\item $x \to y \leq \neg y \to \neg x$,
\item $\neg (x \vee y) = \neg (\neg \neg x \vee \neg \neg y) = \neg x \wedge \neg y$,
\item $\neg \neg (x \wedge y) = \neg \neg x \wedge \neg \neg y$.
\end{enumerate}
\end{proposition}

If $L$ is a ICRDL-algebra, then the set of regular elements of $L$ is $Reg(L) = \{ \neg \neg a \colon a \in A \}$. For any $\star \in \{ \vee, \wedge, \cdot, \to \}$, we consider the operation $x \star_{r} y := \neg \neg (x \star y)$. So, the structure $\langle Reg(L), \vee_{r}, \wedge_{r}, \cdot_{r}, \to_{r}, 0, 1 \rangle$ is a ICRDL-algebra. Note that $a \to_{r} b = a \to b$ and $a \wedge_{r} b = a \wedge b$, for all $a,b \in Reg(L)$. However, $\vee_{r}$ and $\cdot_{r}$ are different, in general, from $\vee$ and $\cdot$, respectively, and so $Reg(L)$ may not be a subalgebra of $L$. Nevertheless, from the results given in \cite{CignoliTorrens}, in some cases $Reg(L)$ can be obtained as a homomorphic image of $L$.

\begin{lemma} \cite{CignoliTorrens}
Let $L$ be a ICRDL-algebra. Then, the following are equivalent:
\begin{enumerate}
\item The correspondence $x \mapsto \neg \neg x$ is a homomorphism of $L$ onto $Reg(L)$,
\item $L$ satisfies the equation 
\[
\neg \neg (\neg \neg x \to x) = 1.
\]
\end{enumerate}
\end{lemma}

Recall that a pseudocomplemented ICRDL-algebra is a ICRDL-algebra $L$ that satisfies the equation
\[
x \wedge \neg x = 0.
\]

\begin{theorem} \cite{Cignoli08,CastaDiazTorrens} \label{R_boolean}
Let $L$ be a ICRDL-algebra. Then $L$ is pseudocomplemented if and only if $Reg(L)$ is a Boolean algebra.
\end{theorem}

\section{Tense ICRDL-algebras} \label{sec3}

In this section, we present an alternative definition of tense operators in ICRDL-algebras. We stress that such a definition differs from the one proposed by Barkhshi in \cite{Bak}.

\subsection{Definition and properties}

In \cite{PelZu23}, tense operators were defined on Heyting algebras and distributive lattices with implication. Given that every Heyting algebra is, in particular, an ICRDL-algebra, it is natural to provide the following definition.

\begin{definition} \label{def2} 
Let $L$ be a ICRDL-algebra. Let $G$, $H$, $F$ and $P$ be unary operations on $L$ satisfying:
\begin{itemize}
\item [(T1)] $P(x) \leq y$ if and only if $x \leq G(y)$,
\item [(T2)] $F(x) \leq y$ if and only if $x \leq H(y)$,
\item [(T3)] $G(0)=0$ and $H(0)=0$,
\item [(T4)] $G(x) \cdot F(y) \leq F(x \cdot y)$ and $H(x) \cdot P(y) \leq P(x \cdot y)$,
\item [(T5)] $G(x \vee y) \leq G(x) \vee F(y)$ and $H(x \vee y) \leq H(x) \vee P(y)$,
\item [(T6)] $G(x \to y) \leq G(x) \to G(y)$ and $H(x \to y) \leq H(x) \to H(y)$. 
\end{itemize}
An algebra ${\mathbf{L}} = \langle L, G, H, F, P \rangle$ will be called tense ICRDL-algebra and $G$,$H$, $F$ and $P$ will be called tense operators.
\end{definition}

\begin{remark} Note that any tense Heyting algebra $\langle A, G, H, F, P\rangle$ (refer to Definition 8.2 in \cite{PelZu23}) is also a tense ICRDL-algebra.
\end{remark}

In the following remarks, we can observe that Definition \ref{def2} and the one proposed by Bakhshi in \cite{Bak} are not necessarily related.

\begin{remark} \label{rem1}

In \cite{Bak}, the author introduces tense operators in the class of residuated lattices as follows: a structure $\langle L, G, H \rangle$ is a tense residuated lattice if $L$ is a residuated lattice, and $G$ and $H$ are unary operations on $L$ satisfying
\begin{enumerate}
\item[(TRL0)] $G(1)=1$ and $H(1)=1$,
\item[(TRL1)] $G(x \to y) \leq G(x) \to G(y)$ and $H(x \to y) \leq H(x) \to H(y)$,
\item[(TRL2)] $x \leq GP(x)$ and $x \leq HF(x)$,
\end{enumerate}
where $F(x) = \neg G( \neg x)$ and $P(x) = \neg H (\neg x)$. If we consider the following bounded distributive lattice
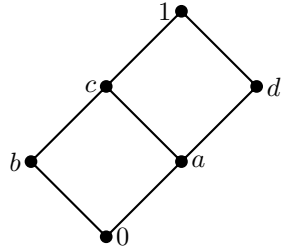
\begin{center}
\begin{tikzpicture}
\filldraw [thick]
(0,1) circle (2pt) node[left] {$b$} --
(1,0) circle (2pt) node[right] {$0$}     -- 
(1,0)--(2,1) circle (2pt) node[right] {$a$} --
(1,2) circle (2pt) node[left] {$c$}--(0,1)
(2,1)--(3,2)circle (2pt) node[right] {$d$}--
(2,3)circle (2pt) node[left]{$1$} -- (1,2);
\end{tikzpicture}
\end{center}
and the operations $\cdot = \wedge$,   
\begin{center}
\begin{tabular}{c|cccccc}
$\to$ & $0$ & $a$ & $b$ & $c$ & $d$ & $1$\tabularnewline
\hline 
$0$ &   $1$   & $1$       & $1$    & $1$ & $1$ & $1$\tabularnewline
$a$ &   $b$   & $1$       & $b$    & $1$ & $1$ & $1$\tabularnewline
$b$ &   $d$   & $d$       & $1$    & $1$ & $d$ & $1$\tabularnewline
$c$ &   $0$   & $d$       & $b$    & $1$ & $d$ & $1$\tabularnewline
$d$ &   $b$   & $c$       & $b$    & $c$ & $1$ & $1$\tabularnewline
$1$ &   $0$   & $a$      & $b$     & $c$ & $d$ & $1$\tabularnewline
\end{tabular}
\hspace{0.5cm}
\begin{tabular}{c|cccccc}
$x$ & $0$ & $a$ & $b$ & $c$ & $d$ & $1$\tabularnewline
\hline 
$G(x)$ &   $0$   & $a$       & $0$    & $c$ & $d$ & $1$\tabularnewline
$H(x)$ &   $0$   & $0$       & $b$    & $c$ & $0$ & $1$\tabularnewline
$F(x)$ &   $0$   & $c$       & $b$    & $c$ & $1$ & $1$\tabularnewline
$P(x)$ &   $0$   & $a$       & $b$    & $c$ & $d$ & $1$\tabularnewline
\end{tabular}
\end{center}
then it is easy to check that $\langle L, G, H, F, P \rangle$ is a tense ICRDL-algebra. However, it is not a tense residuated lattice in the sense of \cite{Bak} since
\[
F(a) = c \neq 1 = \neg G(\neg a),
\] 
and  
\[
P(c) = c \neq 1 = \neg H(\neg c).
\]
\end{remark}

\begin{remark} \label{rem2}
Consider the chain of four elements $L = \{ 0, a, b, 1 \}$, where $0 < a < b < 1$, along with the binary operations $\cdot$ and $\to$ defined by the following tables: 
\begin{center}
\begin{tabular}{c|cccc}
$.$ & $0$ & $a$ & $b$ & $1$ \tabularnewline
\hline 
$0$ &   $0$   & $0$       & $0$    & $0$ \tabularnewline
$a$ &   $0$   & $0$       & $0$    & $a$ \tabularnewline
$b$ &   $0$   & $0$       & $a$    & $b$ \tabularnewline
$1$ &   $0$   & $a$       & $b$    & $1$ \tabularnewline
\end{tabular}
\hspace{0.5cm}
\begin{tabular}{c|cccc}
$\to$ & $0$ & $a$ & $b$ & $1$ \tabularnewline
\hline 
$0$ &   $1$   & $1$       & $1$    & $1$ \tabularnewline
$a$ &   $b$   & $1$       & $1$    & $1$ \tabularnewline
$b$ &   $a$   & $b$       & $1$    & $1$ \tabularnewline
$1$ &   $0$   & $a$       & $b$    & $1$ \tabularnewline
\end{tabular}
\end{center}
We define the tense operators $G$ and $H$ as $G(x) = H(x) = 1$ for all $x \in \{0, a, b, 1\}$. Then the structure $\langle L, G, H \rangle$ is a tense residuated lattice in the sense of \cite{Bak}, but it is not a tense ICRDL-algebra since $H(0) = 1 \neq 0$ and $G(0) = 1 \neq 0$.
\end{remark}

The following proposition can be proved in a similar fashion   to Proposition 3.1 in \cite{PelZu23}, so we leave the details to the reader.

\begin{proposition} 
Let ${\bf{L}}$ be a tense ICRDL-algebra. Then:
\begin{itemize}
\item [(T7)] $G(x \to y) \leq F(x) \to F(y)$ and $H(x \to y)\leq P(x) \to P(y)$,
\item [(T8)] $G(1) = 1$ and $H(1) =1$,
\item [(T9)] $G( x \wedge y) = G(x) \wedge G(y)$ and $H( x \wedge y) = H(x) \wedge H(y)$,
\item [(T10)] $x \leq GP(x)$ and $x \leq HF(x)$,
\item [(T11)] $F(0) = 0$ and $P(0) = 0$, 
\item [(T12)] $F( x \vee y) = F(x) \vee F(y)$ and $P(x \vee y) = P(x) \vee P(y)$, 
\item [(T13)] $FH(x) \leq x$ and $PG(x) \leq x$,
\item [(T14)] $x \leq y$ implies $G(x) \leq G(y)$ and $H(x) \leq H(y)$, 
\item [(T15)] $x \leq y$ implies $F(x) \leq F(y)$ and $P(x)\leq P(y)$,
\item [(T16)] $x \cdot F(y) \leq F(P(x) \cdot y)$ and $x \cdot P(y) \leq P(F(x) \cdot y)$, 
\item [(T17)] $F(x) \cdot y = 0$ if and only if $x \cdot P(y) = 0$, 
\item [(T18)] $G (x \vee H(y)) \leq G(x) \vee y$ and $H(x \vee G(y)) \leq H(x) \vee y$, 
\item [(T19)] $x \vee H(y) = 1$ if and only if $G(x) \vee y = 1$.
\end{itemize}
\end{proposition}

The following result was proved for Heyting algebras in \cite{FP}. We now present its version for ICRDL-algebras.

\begin{lemma} Let $G$ and $H$ be two unary operators on a ICRDL-algebra ${\bf{L}}$ such that $G$ and $H$ satisfy {\rm (T14)}. Then {\rm (T6)} is equivalent to
\begin{itemize}
\item [(T6')] $G(x) \cdot G(y) \leq G( x \cdot y)$ and $H(x) \cdot H(y) \leq H(x\cdot y)$.
\end{itemize}
\end{lemma}

\begin{proof}
Let $a,b \in L$ and suppose that (T6') is true. Then 
\[
G(a) \cdot G(a \to b) \leq G(a \cdot (a \to b)) \leq G(b).
\]
So, $G(a \to b) \leq G(a) \to G(b)$. Reciprocally, if we suppose that (T6) is true, then we have $G(a) \leq G(b \to (a \cdot b)) \leq G(b) \to G(a \cdot b)$ and $G(a) \cdot G(b) \leq G(a \cdot b)$. Is similar for the operator $H$.
\end{proof}

\begin{remark} 
Note that if ${\bf{L}}$ is a tense ICRDL-algebra, then is a fuzzy dynamic algebra in the sense of \cite{ChajdaPaseka}.  
\end{remark}

\begin{proposition} \label{propo_pro}
Let $\langle L, G, H, F, P \rangle$ be a tense ICRDL-algebra. Then
\[
G(x) \cdot F(x \to y) \leq F(y) \text{ and } H(x) \cdot P(x \to y) \leq P(y).
\]
\end{proposition}
\begin{proof}
Let $a,b \in L$.  Since $a \leq (a \to b) \to b$, then 
\[
G(a) \leq G((a \to b) \to b) \leq F(a \to b) \to F(b)
\]
and $G(a) \cdot F(a \to b) \leq F(b)$. Similarly, we can prove the inequality $H(x) \cdot P(x \to y) \leq P(y)$. 
\end{proof}

\subsection{Connections between tense ICRDL-algebras and tense Boolean algebras}

We know that if $\bf{L}$ is a ICRDL-algebra, then the set of regular elements $Reg(L)$ of $\bf{L}$ is a ICRDL-algebra. We conclude this section by showing some necessary and sufficient conditions for the ICRDL-algebra $Reg(L)$ to be a tense Boolean algebra.

The following result was proven in \cite{FigPel14}, within the context of Heyting algebras.

\begin{proposition}\label{propo_neg}
Let ${\bf{L}}$ be a tense ICRDL-algebra. Then:
\begin{enumerate}
\item $G(\neg x) \leq \neg G(x)$ and $H(\neg x) \leq \neg H(x)$,
\item $G(\neg x) \leq \neg F(x)$ and $H(\neg x) \leq \neg P(x)$,
\item $G(x) \leq \neg F( \neg x)$ and $H(x) \leq \neg P( \neg x)$,
\item $F(\neg x) \leq \neg G(x)$ and $P(\neg x) \leq \neg H(x)$.
\end{enumerate}
\end{proposition}
\begin{proof}
We only prove (4). By Proposition \ref{propo_pro}, we have $G(x) \cdot F(x \to y) \leq F(y)$. Consequently, according to (T11), it follows that $G(x) \cdot F(x \to 0) \leq F(0) = 0$, and $F(\neg x) \leq G(x) \to 0$, i.e., $F(\neg x) \leq \neg G(x)$.
\end{proof}

If ${\bf{L}}$ is a tense ICRDL-algebra, then we know that the structure $Reg( \mathbf{L} ) = \langle Reg(L), \vee_{r}, \wedge_{r}, \cdot_{r}, \to_{r}, 0, 1 \rangle$ is a ICRDL-algebra. Now, we define on $Reg(L)$ the following unary operators 
\begin{eqnarray*}
G_{r}(x)   & := & \neg \neg G(x), \\
H_{r}(x)   & := & \neg \neg H(x), \\
F_{r}(x)   & := & \neg \neg F(x), \\
P_{r}(x)   & := & \neg \neg P(x), 
\end{eqnarray*}
and consider the inequalities 
\begin{itemize}
\item[(A1)] $\neg F (\neg x) \leq G(x)$, 
\item[(A2)] $\neg P (\neg x) \leq H(x)$.
\end{itemize}

\begin{remark} \label{re_A1-A2}
Not every tense ICRDL-algebra satisfies conditions $(A1)$ and $(A2)$ (see \cite{FigPel14}, Example 2). On the other hand, by Proposition \ref{propo_neg}, in every tense ICRDL-algebra  that satisfies conditions {\rm{(A1)}} and {\rm{(A2)}}, we have the equations $G(x) = \neg F (\neg x)$ and $H(x) = \neg P (\neg x)$.
\end{remark}

\begin{theorem}
Let ${\bf{L}}$ be a tense ICRDL-algebra. Then, there are equivalent:
\begin{enumerate}
\item $\bf{L}$ is pseudocomplemented and satisfies  conditions (A1) and (A2), 
\item $\langle Reg( \mathbf{L} ), G_{r}, H_{r} \rangle$ is a tense Boolean algebra.
\end{enumerate}
\end{theorem}
\begin{proof}
$(1) \Rightarrow (2)$ By Theorem \ref{R_boolean}, the structure $Reg( \mathbf{L} )$ is a Boolean algebra. By (T8) and (T9) it easily follows (TB1) and (TB2). We prove (TB3). Let $a \in Reg(L)$. Then, by Remark \ref{re_A1-A2} and as $a = \neg \neg a$, we have 
\[
\neg G_{r}(\neg a) = \neg \neg \neg G(\neg a) = \neg G(\neg a) = \neg \neg F(\neg \neg a) = \neg \neg F(a) = F_{r}(a),
\]
i.e., $F_{r}(a) = \neg G_{r}(\neg a)$. Similarly, we have $P_{r}(a) = \neg H_{r}(\neg a)$. Finally, we see that $a \leq G_{r}P_{r}(a)$. By (T10), (T14), (T15) and (2) of Proposition \ref{prop_neg},
\[
a \leq GP(a) \leq \neg \neg G( \neg \neg P(a)) = G_{r}P_{r}(a).
\]   
Analogously, $a \leq H_{r}F_{r}(a)$. So, $\langle Reg( \mathbf{L} ), G_{r}, H_{r} \rangle$ is a tense Boolean algebra. 

$(2) \Rightarrow (1)$ It is immediate.
\end{proof}

\subsection{Tense DRL-algebras} \label{sec4}

The aim of this subsection is to introduce the class of tense DRL-algebras, which will be crucial for the rest of the paper.

We recall that a centered integral involutive residuated lattice is a structure $\langle A, \vee, \wedge, \ast, \sim, c, 0, 1 \rangle$ such that: 
\begin{enumerate}
\item $\langle A, \vee, \wedge, \ast, \sim, \Rightarrow, 0,1 \rangle$ is an integral commutative residuated lattice, where $x \Rightarrow y = \sim (x \ast (\sim y))$, 
\item $\sim$ is a dual lattice-automorphism which is also an involution, 
\item $c$ is a fixed point of the involution.
\end{enumerate}

If there is no risk of confusion, we will denote centered integral involutive residuated lattices by their underlying set $A$.

A centered integral involutive residuated lattice $A$ is said to be a {\it{c-differential residuated lattice}}, or {\it{DRL-algebra}} for short, if the following equation, called {\it{Leibniz condition}} in \cite{CMS,CLS}, holds:
\begin{itemize}
\item[(LC)] $(x \ast y) \wedge c = ((x \wedge c) \ast y) \vee (x \ast (y \wedge c))$. 
\end{itemize}

The following lemma will be useful in Section \ref{sec7} of this paper.

\begin{lemma} \label{lem: technical DRL}
Let $A$ be a DRL-algebra. Then, the following hold: 
\begin{enumerate}
\item $[(x\vee c)\ast(y\vee c)] \vee c =(x\ast y) \vee c.$
\item $\sim ((x\vee c)\ast ( \sim (\sim y\vee c)))\wedge \sim ((y\vee c)\ast (\sim (\sim x\vee c)))= \sim (x\ast y)\vee c.$
\end{enumerate}
\end{lemma}
\begin{proof}
On the one hand, observe that 
\begin{displaymath}
\begin{array}{lcl}
[(x\vee c)\ast(y\vee c)] \vee c & = & (x\ast y) \vee (c\ast y) \vee (x\ast c) \vee (c\ast c) \vee c  \\
                                               & = & (x\ast y) \vee c.
\end{array}
\end{displaymath}
On the other hand, let 
\[
t(x,y)=\sim ((x\vee c)\ast ( \sim (\sim y\vee c)))\wedge \sim ((y\vee c)\ast (\sim (\sim x\vee c))).
\]
Then, we have  
\begin{displaymath}
\begin{array}{lcl}
t(x,y) & = & \sim [((x\vee c)\ast (y\wedge c))\vee ((y\vee c)\ast (x\wedge c))]  \\
         & = & \sim [(x\ast (y\wedge c))\vee (y\ast (x\wedge c))\vee (c\ast (x\wedge c))\vee (c\ast (y\wedge c))].
\end{array}
\end{displaymath} 
Note that $c\ast(x\wedge c)=\sim (c\Rightarrow (\sim x\vee c))=\sim 1 =0$. Similarly $c\ast(y\wedge c)=0$. Finally, by (LC), $(x\ast (y\wedge c))\vee (y\ast (x\wedge c))=(x \ast y) \wedge c$. Hence, 
\[
t(x,y) = \sim [((x \ast y) \wedge c) \vee 0 ] =  \sim (x \ast y) \vee c. 
\]
This concludes the proof.
\end{proof}

Let $A$ be a DRL-algebra and $G,H \colon A \rightarrow A$ be two unary operators. We define the operators $F$ and $P$ by $F(x) :=\sim G(\sim x)$ and $P(x) :=\sim H (\sim x)$, for any $x \in A$. 

\begin{definition}\label{def: tDRL}
Let $A$ be a DRL-algebra. Let $G$ and $H$ be unary operations on $A$ satisfying:
\begin{itemize}
\item [(t0)] $G(1)=1$ and $H(1)=1,$ 
\item [(t1)] $G(c)=c$ and $H(c)=c$,
\item [(t2)] $G(x \wedge y) = G(x) \wedge G(y)$ and $H(x \wedge y) = H(x) \wedge H(y)$,
\item [(t3)] $x \leq GP(x)$ and $x \leq HF(x)$,
\item [(t4)] $G(x \vee y) \leq G(x) \vee F(y)$ and $H(x \vee y) \leq H(x) \vee P(y)$,
\item [(t5)] $G(x \Rightarrow y) \leq G(x) \Rightarrow G(y)$ and $H(x \Rightarrow y) \leq H(x) \Rightarrow H(y)$.
\end{itemize}
An algebra ${\bf{A}} = \langle A, G, H \rangle$ will be called tense DRL-algebra.
\end{definition}

\begin{remark}
Notice that if ${\bf A}$ is a tense DRL-algebra, then the reduct $\langle A, \vee, \wedge, \sim, G, H, 0, 1\rangle$ becomes a tense De Morgan algebra, as defined in \cite{FP14}.    
\end{remark}

The following result enables us to provide an equivalent definition of tense DRL-algebras.

\begin{lemma} \label{Algunos ahorros}
Let ${\bf{A}}$ be a tense DRL-algebra. Then, axiom {\rm(t5)} is equivalent to each of the following axioms: 
\begin{itemize}
\item [(t6)] $G(x \Rightarrow y) \leq F(x) \Rightarrow F(y)$ and $H(x \Rightarrow y) \leq P(x) \Rightarrow P(y)$.
\item [(t7)] $G(x) \ast F( \sim y) \leq F(x \ast (\sim y))$ and $H(x) \ast P(\sim y) \leq P(x \ast (\sim y))$,
\item [(t8)] $F(x) \ast G(y) \leq F(x \ast y)$ and $P(x) \ast H(y) \leq F(x \ast y)$.
\end{itemize}
\end{lemma}

\begin{proof}
Let $a,b \in A$. We only prove that ${\rm (t6)}$ is equivalent to ${\rm (t7)}$. We start by noticing that the following follows straight from the definitions: $a \ast (\sim b) = \sim (a \Rightarrow b)$ and $\sim F(a) = G(\sim a)$. Therefore, if $G(a \Rightarrow b) \leq G(a) \Rightarrow G(b)$, then 
\[
G( \sim (a \ast (\sim b)))\leq G(a) \Rightarrow G(b).
\]
Thus, since $\sim$ is antitone, from the identities above we get $\sim (G(a) \Rightarrow G(b)) \leq F(a \ast (\sim b))$. Hence, $G(a) \ast F(\sim b) \leq F(a \ast (\sim b))$. The converse uses the same arguments. The details of the remaining proofs are left to the reader.
\end{proof}

The following proposition will be useful in the categorical equivalence that we will present in the next section.

\begin{proposition} \label{propiedades tDRLc}
Let ${\bf{A}}$ be a tense DRL-algebra. Then:
\begin{itemize}
\item [{(c1)}] $F(c) = c$ and $P(c) = c$, 
\item [{(c2)}] $G(x \vee c) = G(x) \vee c$ and $H(x \vee c) = H(x) \vee c$, 
\item [{(c3)}] $F(x \wedge c) = F(x) \wedge c$ and $P(x \wedge c) = P(x) \wedge c$.
\end{itemize}
\end{proposition}
\begin{proof} 
We will only prove (c1) and (c2).

(c1): By (t1) and (t3), $F(c) = \sim G(\sim c) = \sim G(c) = \sim c = c$. Similarly, $P(c)=c$.

(c2): Let $a \in A$. Since $a \leq a \vee c$ and $c \leq a \vee c$, we have that $G(a) \leq G(a \vee c)$ and $c = G(c) \leq G(a \vee c)$. So, $G(a) \vee c \leq G(a \vee c).$ On the other hand, from (t4) and (c1), we obtain $G(a \vee c) \leq G(a) \vee F(c) = G(a) \vee c$. Therefore, $G(a \vee c) = G(a) \vee c$. Similarly, we have $H(a \vee c) = H(a) \vee c$.
\end{proof}

\section{Kalman's construction} \label{sec5}

In this section we prove some results that establish the connection between the class of tense ICRDL-algebras and the class of tense DRL-algebras.

Let ${\mathbf{L}}$ be a tense ICRDL-algebra and let us consider
\[
K(L) = \{ (a,b) \in L \times L \colon a\cdot b = 0 \}.
\]
It is well known from \cite{CMS,CLS} that by defining
\begin{eqnarray*}
(a,b) \vee (x,y).           & := & (a \vee x, b \wedge y), \\ 
(a,b) \wedge (x,y)       & := & (a \wedge x, b \vee y), \\
(a,b) \ast (x,y)             & := & (a \cdot x, (a \to y) \wedge (x \to b)), \\
(a,b) \Rightarrow (x,y) & := & ((a \to x) \wedge (y \to b), a \cdot y), \\
                     \sim (a,b)& := & (b,a), \\
                                 0 & := & (0,1), \\
                                 1 & := & (1,0), \\
                                 c & := &(0,0),
\end{eqnarray*}
we get that the structure $ \mathbf{L}_{K} = \langle K(L), \vee, \wedge, \ast, \sim, c, 0, 1 \rangle$ is a DRL-algebra. Now, we define on $K(L)$ the following unary operators given by 
\begin{eqnarray*}
G_K((a,b))   & := & (G(a),F(b)), \\
H_K((a,b))   & := & (H(a),P(b)), \\
F_{K}((a,b)) & := & (F(a),G(b)), \\
P_{K}((a,b)) & := & (P(a),H(b)).
\end{eqnarray*}

Taking into account the unary operators defined earlier, we obtain the following result.

\begin{lemma} \label{lf} 
Let ${\mathbf{L}}$ be a tense ICRDL-algebra. Then, for every $(a,b) \in K(L)$:
\begin{enumerate}
\item $G_{K}(a,b), H_{K}(a,b) \in K(L)$, 
\item $F_{K}(a,b) = \sim G_{K}( \sim (a,b))$ and $P_{K}(a,b) = \sim H_{K} (\sim (a,b))$,
\item $F_{K}(a,b), P_{K}(a,b) \in K(L)$.
\end{enumerate}
\end{lemma}
\begin{proof} We will only prove (1). If $(a,b) \in K(A)$, then $a \cdot b=0$. Therefore, from (T4) and (T11), $G(a) \cdot F(b) \leq F(a \cdot b) = F(0) = 0$. So, $(G(a), F(b)) \in K(A)$. The proof of $H_{K} (a,b) \in K(A)$ is similar. The details of the proof of the remaining items are left to the reader.
\end{proof}

\begin{lemma} \label{kalmanH} 
Let ${\mathbf{L}}$ be a tense ICRDL-algebra. Then
\[
K( \mathbf{L} ) = \langle \mathbf{L}_{K}, G_{K}, H_{K} \rangle
\]
is a tense DRL-algebra. 
\end{lemma}
\begin{proof} From Lemma \ref{lf}, the operators $G_{K}$ and $H_{K}$ are well defined. We will prove the axioms of tense DRL-algebra. Due to the symmetry of
tense operators $G_K$ and $H_K$, we will prove the axioms only for the operator $G_{K}$. We take $(a,b), (x,y)\in K(A)$.

(t0):  From (T8) and (T11) it follows  
\[
G_K(1) = G_K (1,0) = (G(1),F(0)) = (1,0) = 1.
\]
Similarly, $H_{K}(1)=1$.

(t1):  From (T3) and (T10) it follows  
\[
G_K(c) = G_K (0,0) = (G(0),F(0)) = (0,0) = c.
\]
Similarly, $H_{K}(c)=c$.

(t2): From (T9) and (T12) we have 
\[
\begin{array}{lll}
G_{K}((a,b) \wedge (x,y))  & = & G_{K}(a \wedge x, b \vee y) \\
                                          & = & (G(a \wedge x),F(b \vee y)) \\
                                          & = & (G(a) \wedge G(x), F(b) \vee F(y)) \\
                                          & = & (G(a), F(b)) \wedge (G(x), F(y)) \\ 
                                          & = & G_{K}(a,b) \wedge G_{K}(x,y). \\ 
\end{array}
\]

(t3): By properties (T10) and (T13), 
\[
\begin{array}{lll}
(a,b) \wedge G_{K}P_{K}(a,b)  & = & (a,b) \wedge G_{K}((P(a), H(b))) \\
                                                 & = & (a,b) \wedge (GP(a), FH(b)) \\
                                                 & = & (a \wedge GP(a), b \vee FH(b)) \\
                                                 & = & (a,b), \\ 
\end{array}
\]
i.e., $(a,b) \leq G_{K}P_{K}(a, b)$.

(t4): From (T5) and (T4), then 
\[
\begin{array}{lll}
G_{K}((a, b) \vee (x,y))  & = & G_{K}(a \vee x, b \wedge y) \\
                                      & = & (G(a \vee x), F(b \wedge y)) \\
                                      & \leq & (G(a) \vee F(x), G(b) \wedge F(y)) \\
                                      & = & (G(a), G(b)) \vee (F(x), F(y)) \\
                                      & = & G_{K}(a, b) \vee F_{K}(x, y). \\ 
\end{array}
\]

(t5): We start by noticing that $(a,b) \Rightarrow (x,y) = \sim ((a,b) \ast (y,x))$. Thus, it is the case that $G_{K}((a,b) \Rightarrow (x,y)) = (G(a \rightarrow x)\wedge G(y \rightarrow b), F(a \cdot y))$. Further, $G_{K}(a, b) \Rightarrow G_{K}(x, y) = ((G(a) \rightarrow G(x)) \wedge (F(y) \rightarrow F(b)), G(a) \cdot F(y))$. So, from (T6) and (T7), we have $G(a \to x) \leq G(a) \to G(x)$ and $G(y \to b) \leq F(y) \to F(b)$. Besides, from (T4), $G(a) \cdot F(y) \leq F(a \cdot y)$. Hence, 
\[
(G(a \to x) \wedge G(y \to b), F(a \cdot y)) \leq ((G(a) \to G(x)) \wedge (F(y) \to F(b)), G(a) \cdot F(y)).
\]
Therefore, $G_{K}((a, b) \Rightarrow (x, y)) \leq G_K(a, b) \Rightarrow G_{K}(x, y)$. 
\end{proof}

We write $\mathsf{tICRDL}$ for the category whose objects are tense ICRDL-algebras, $\mathsf{tDRL}$ for the category whose objects are tense DRL-algebras and in both cases, the morphisms are the corresponding algebra homomorphisms. If $\mathbf{L} = \langle L, G, H, F, P \rangle$ and $\mathbf{M} = \langle M, G, H, F, P \rangle$ are tense ICRLD-algebras and $f \colon L \rightarrow M$ is a morphism in $\mathsf{tICRDL}$, then it is no hard to see that the map $K(f) \colon K(L) \rightarrow K(M)$ given by $K(f)(x, y) = (f(x), f(y))$ is a morphism in $\mathsf{tDRL}$ from $\langle K(\mathbf{L}), G_{K}, H_{K} \rangle$ to $\langle K(\mathbf{M}), G_{K}, H_{K} \rangle$. These assignments establish a functor $K$ from $\mathsf{tICRDL}$ to $\mathsf{tDRL}$.

Let ${\mathbf{A}}$ be a DRL-algebra and define 
\[
C(A) = \{ x \in A \colon x\geq c\}.
\]
If we consider $x \cdot y = (x \ast y) \vee c$ and $x \to y = \sim (x \ast (\sim y))$, then the structure
\[
\mathbf{A}_C = \langle C(A), \vee, \wedge, \cdot, \to, c, 1 \rangle
\]
is a ICRDL-algebra (see \cite{CMS}, Theorem 7.6). Moreover, if $f \colon A \rightarrow B$ is a homomorphism of DRL-algebras, then it is not hard to see that $C(f) \colon C(A)\rightarrow C(B)$, defined by $C(f)(x) = f(x)$, is a homomorphism of DRL-algebras. The following lemma shows that the latter construction can be lifted to the class of tense DRL-algebras.

\begin{lemma} \label{tDRL to tICRL} 
Let ${\mathbf{A}}$ be a tense DRL-algebra. Then
\[
C(\mathbf{A}) = \langle \mathbf{A}_C, G, H, F, P \rangle
\]
is a tense ICRDL-algebra. Moreover, if $f \colon A \rightarrow B$ is a morphism in $\mathsf{tDRL}$, then $C(f) \colon C(A) \rightarrow C(B)$ is a morphism in $\mathsf{tICRDL}$.
\end{lemma}
\begin{proof}
Let $a,b \in A$. We only proof (T4). Notice that from (t3), it is the case that $F$ preserves all the existing joins. Therefore, from by Lemma \ref{Algunos ahorros} we have 
\[ 
G(x) \cdot F(y) = (G(x) \ast F(y)) \vee c \leq F(x \ast y) \vee c = F(x \cdot y),
\]  
as desired.
\end{proof}

Let ${\mathbf{A}}$ and ${\mathbf{B}}$ be two tense DRL-algebras. Let $f \colon A \rightarrow B$ be a homomorphisms of tense DRL-algebras. It is clear now, from Lemma \ref{tDRL to tICRL}, that the assingments $A \mapsto C(\mathbf{A})$ and $f \mapsto C(f)$ determine a functor $C$ from $\mathsf{tDRL}$ to $\mathsf{tICRDL}$.

\begin{lemma}\label{lemma alpha map}
Let ${\mathbf{L}}$ be a tense ICRDL-algebra. Then the map $\alpha_{L} \colon L \rightarrow C(K(L))$ given by $\alpha_{L}(x) = (x, 0)$ is an isomorphism in $\mathsf{tICRDL}$.
\end{lemma}
\begin{proof} 
Taking into account \cite[Theorem 7.6]{CMS}, we only have to prove that $\alpha_{L}$ preserves the tense operators. Let $a \in A$. Then 
\[
\alpha_{L}(G(a)) = (G(a), 0) = (G(a), F(0)) = G_{K}((a, 0)) =G_{K}(\alpha_{L}(a))
\]
and
\[
\alpha_{L}(F(a)) = (F(a), 0) = (F(a), G(0)) = F_{K}((a, 0)) = F_{K}(\alpha_{L}(a)).
\]
Similarly, we can prove $\alpha_{L}(H(a)) = H_{K}(\alpha_{L}(a))$ and $\alpha_{L}(P(a)) = P_{K}(\alpha_{L}(a))$.
\end{proof}

If $\mathbf{A}$ is a DRL-algebra, then 
\[
\beta_{A} \colon A \rightarrow K(C(A))
\]
given by  $\beta_{A}(x) = (x \vee c, \sim x \vee c)$ is an injective homomorphism of DRL-algebras (see \cite{CMS}, Lemma 7.5). We consider the folllowing condition on the class of DRL-algebras:
\begin{itemize}
\item [(CK$^{\cdot}$)] for every $x,y \geq c$ such that $x \cdot y \leq  c$, there exists $z \in A$ such that $z \vee c=x$ and $\sim z \vee c=y$.
\end{itemize}

The next lemma follows from the definition of $\beta_{A}$.

\begin{lemma} \cite{CMS} \label{CMS_lemma}
Let $\mathbf{A}$ be a DRL-algebra. Then $\mathbf{A}$ satisfies the condition {\rm (CK$^{\cdot}$)} if and only if $\beta_{A}$ is a surjective map. 
\end{lemma}

Next, we will present a crucial lemma for the proof of Theorem \ref{equivalence theorem}.

\begin{lemma} 
Let ${\mathbf{A}}$ be a tense DRL-algebra. Then the map $\beta_{A}$ is a monomorphism in $\mathsf{tDRL}$.
\end{lemma}
\begin{proof} 
Let $a \in A$.  From (t3), (c2) and (c1) we have that 
\[
\begin{array}{lll}
\beta_{A}(G(a)) & = & (G(a) \vee c, \sim G(a) \vee c) \\
                         & = & (G(a) \vee c, F(\sim a) \vee c) \\
                         & = & (G(a \vee c), F(\sim a \vee c)) \\
                         & = & G_{K}((a \vee c, \sim a \vee c)), \\
\end{array}
\]
i.e., $\beta_{A}(G(a)) = H_{K}(\beta_{A}(a))$. Similarly, $\beta_{A}(H(a)) = H_{K}(\beta_{A}(a))$.
\end{proof}

We will denote by $\mathsf{tDRL}_{c}$ for the full subcategory of $\mathsf{tDRL}$ whose objects satisfy the condition (CK$^{\cdot}$). Straightforward computations based on previous results of this section prove the following result.

\begin{theorem} \label{equivalence theorem}
The functors K and C establish a categorical equivalence between $\mathsf{tICRDL}$ and $\mathsf{tDRL_{c}}$ with natural isomorphisms $\alpha$ and $\beta$.
\end{theorem}

\section{A round trip between tense filters} \label{sec6}

Let $\mathbf{L}$ be a ICRDL-algebra.  In Remark 7.12 of \cite{CMS} it was proved that Kalman's construction can be used in order to establish an isomorphism between the congruence lattice of $\mathbf{L}$ and the congruence lattice of the DRL-algebra $K(\mathbf{L})$. Moreover, if $\mathbf{A}$ is a DRL-algebra, such a result can also be extended for proving that there is an isomorphism between the congruence lattice of $\mathbf{A}$ and the congruence lattice of the ICRDL-algebra $C(\mathbf{L})$. In this section we are intended to extend these results to the context of tense DRL-algebras and tense ICRDL-algebras. Nevertheless, our approach borrows some of the ideas employed along Section 7.1 of \cite{PelZu23}. As an application, we provide a detailed description between the tense filters of the members of the varieties of tense DRL-algebras and tense ICRDL-algebras.

We start by recalling some notions. Let $\mathbf{L}$ be a bounded distributive lattice. If $\theta\in \mathsf{Con}(\mathbf{L})$, we can define a congruence $\gamma_{\theta}$ of $K(\mathbf{L})$  by
\[
(a,b) \gamma_{\theta}(x,y)\ \text{ if and only if }\ (a,x)\in \theta\ \text{ and } (b,y)\in \theta.
\]
Reciprocally, if $\gamma \in \mathsf{Con}(K(\mathbf{L}))$, we can also define a congruence $\theta^{\gamma}$ of $\mathbf{L}$ as 
\[
(a,b) \in \theta^{\gamma} \text{ if and only if } (a,0) \gamma (b,0).
\]

\begin{lemma}\label{iso gamma}
Let $\mathbf{L}$ be a tense ICRDL-algebra and let $\theta \in \mathsf{Con}(\mathbf{L})$. Then $\gamma_{\theta} \in \mathsf{Con}(K(\mathbf{L}))$.
\end{lemma}
\begin{proof}
We only proof that $\gamma_{\theta}$ is compatible with the product because the proof of the compatibility with respect the residual is analogue. The proof of the compatibility of $\theta^{\gamma}$ with the rest of the operations uses the same ideas of Lemma 7.1 of \cite{PelZu23}.  Let $(a_i,b_i)\gamma_{\theta}(x_i,y_i)$ with $i=1,2$. By definition of $\gamma_{\theta}$ we have that $(a_i,x_i),(b_i,y_i)\in \theta$ for $i=1,2$. Then, since $\theta$ is a congruence of $\mathbf{L}$ we get that $(a_{1}a_{2},x_{1}x_{2}), ((a_{1}\rightarrow b_{2})\wedge (a_{2}\rightarrow b_{1}),(x_{1}\rightarrow y_{2})\wedge (x_{2}\rightarrow y_{1}))$ belong to $\theta$. Therefore, $(a_1,b_1)\ast (a_2,b_2)\gamma_{\theta} (x_1,y_1)\ast (x_2,y_2)$.
\end{proof}

\begin{lemma} \label{iso theta}
Let $\mathbf{L}$ be a tense ICRDL-algebra and let  $\gamma \in \mathsf{Con}(K(\mathbf{L}))$. Then $\theta^{\gamma} \in \mathsf{Con}(\mathbf{L})$.
\end{lemma}
\begin{proof}
We will prove the compatibility of $\theta^{\gamma}$ with respect to the implication and the product. For the compatibility of $\theta_{\gamma}$ with the rest of the operations, the reader may consult Lemma 7.1 of \cite{PelZu23}. To do so, let $(a_1,b_1),(a_2,b_2)\in \theta^{\gamma}$. By definition of $\theta^{\gamma}$ it is the case that $(a_i,0)\gamma(b_i,0)$ with $i=1,2$. Thus, $(a_1\rightarrow a_2,0)\gamma(b_1\rightarrow b_2,0)$ from the assumption on $\gamma$. Hence, $(a_1,0)\Rightarrow (a_2,0)\gamma (b_1,0)\Rightarrow (b_2,0),$ so  $\theta^{\gamma}$ is compatible with the implication. On the other hand, since $\gamma\in \mathsf{Con}(K(\mathbf{L}))$, we obtain that $(a_1,0)\ast (a_2,0)\gamma (b_1,0)\ast (b_2,0)$ and $(a_{1}a_{2},\neg a_1 \wedge \neg a_2)\gamma (b_{1}b_{2},\neg b_1 \wedge \neg b_2)$. Moreover, we also get that $(0,a_2)=\sim (a_2,0)\gamma \sim (b_2,0)=(0,b_2),$ so $(a_1,0)\Rightarrow (0,a_2)\gamma (b_1,0)\Rightarrow (0,b_2)$ and consequently $(\neg a_1 \wedge \neg a_2,0)\gamma (\neg b_1 \wedge \neg b_2,0)$. Finally, by transitivity on $\gamma$ we get $(a_{1}a_{2},0)\gamma (b_{1}b_{2},0)$, we can conclude $(a_{1}a_{2},b_{1}b_{2})\in \theta^{\gamma}$, as desired.
\end{proof}

The proof of the following result is a straight consequence of Lemmas \ref{iso gamma} and \ref{iso theta} and it uses the same ideas deployed in Corollary 5.4 of \cite{CasCeSan}. So, we leave the details to the reader.

\begin{theorem}\label{iso congruencias via funtor K}
Let $\mathbf{L}$ be a tense ICRDL-algebra. Then the assignment $f:\mathsf{Con}(\mathbf{L})\rightarrow \mathsf{Con}(K(\mathbf{L}))$ defined by $f(\theta)=\gamma_\theta$ is an order isomorphism.
\end{theorem}

\begin{definition}
Let $\mathbf{L}$ be a tense ICRDL-algebra. A non-empty subset $S$ of $L$ is a tense filter provided that $S$ is an up-set, $S$ is closed under multiplication, $G$ and $H$.  
\end{definition}

\begin{remark}\label{vale} 
We claim that tense filters of tense ICRDL-algebras are closed under $F$ and $P$. Indeed, if $S$ is a tense filter of a tense ICRDL-algebra $\mathbf{L}$ and $x\in S$, then since $G(0)=0$ and $H(0)=0$, by {\rm (T5)} we get that $G(x)\leq F(x)$ and $H(x)\leq P(x)$. Since $S$ is increasing and closed under $G$ and $H$, then $F(x), P(x)\in S$, as claimed.
\end{remark}

Let $\mathbf{L}$ be a tense ICRDL-algebra. We write $\mathrm{tFi}(\mathbf{L})$ for the poset of tense filters of $\mathbf{L}$ orderded by inclusion and $\mathsf{Con}(\mathbf{L})$ for the congruence lattice of $\mathbf{L}$. If we consider the operation $x\circ y=(x\rightarrow y)\cdot (y\rightarrow x)$, then we have the following result.

\begin{theorem} \label{congruences tICRDL}
Let $\mathbf{L}$ be a tense ICRDL-algebra, $S\in \mathrm{tFi}(\mathbf{L})$ and $\theta\in \mathsf{Con}(\mathbf{L})$. Then: 
\begin{enumerate}
\item $S_{\theta}:=1/\theta$ is a tense filter of $\mathbf{L}$,
\item $\theta_{S}:=\{(x,y)\in L^{2} \colon x\circ y\in S\}$ is a congruence of $\mathbf{L}$,
\item The assignments $\theta\mapsto S_{\theta}$ and $S\mapsto \theta_{S}$ establish a mutually-inverse poset isomorphism between $\mathrm{tFi}(\mathbf{L})$ and $\mathsf{Con}(\mathbf{L})$.
\end{enumerate}
\end{theorem}
\begin{proof}
It is proved in a similar way to Lemma 11 of \cite{FigPel14}.    
\end{proof}

Let $\mathbf{A}$ be a tense DRL-algebra. Note that since by definition, $\mathbf{A}$ is a tense ICRDL-algebra, and by Theorem \ref{congruences tICRDL}, we can obtain a description of the lattice of congruences of $\mathbf{A}$ by means of tense filters.

\begin{corollary} \label{congruencias tDRL}
Let $\mathbf{A}$ be a tense DRL-algebra. Then the assignments $\theta\mapsto D_{\theta}$ and $D\mapsto \theta_{D}$ establish a mutually-inverse poset isomorphism between $\mathrm{tFi}(\mathbf{A})$ and $\mathsf{Con}(\mathbf{A})$. 
\end{corollary}

Let $\mathbf{A}$ be a tense DRL-algebra.  Our next goal is to describe explicitly how from a congruence of $\mathbf{A}$ we can get a unique tense filter of $C(\mathbf{A})$. We start by recalling from Theorem \ref{equivalence theorem} that the map $\beta_{\mathbf{A}} \colon \mathbf{A}\rightarrow K(C(\mathbf{A}))$ defined by $\beta_{\mathbf{A}}(x) = (x\vee c, \sim x \vee c)$ is an isomorphism of tense DRL-algebras. Further, from general reasons the map $h_{\mathbf{A}} = \beta_{\mathbf{A}} \times \beta_{\mathbf{A}}$ which is defined by $h_{\mathbf{A}}(x,y) = (\beta_{\mathbf{A}}(x), \beta_{\mathbf{A}}(y))$, induces an isomorphism between $\mathsf{Con}_{\mathrm{tDRL}}(\mathbf{A})$ and $\mathsf{Con}_{\mathrm{tDRL}}(KC(\mathbf{A}))$. Now, from Theorem \ref{iso congruencias via funtor K}, the map 
\[
g \colon \mathsf{Con}_{\mathrm{tDRL}}(K(C(\mathbf{A}))) \rightarrow \mathsf{Con}_{\mathrm{tICDRL}}(C(\mathbf{A}))
\] 
defined by $g(\gamma)=\theta^{\gamma}$ is an isomorphism. Consider the composite $w = gh_{\mathbf{A}}$, i.e., 
\[
w \colon \mathsf{Con}_{\mathrm{tDRL}}(\mathbf{A})\rightarrow  \mathsf{Con}_{\mathrm{tICDRL}}(C(\mathbf{A})).
\]

\begin{lemma}\label{Fil(A) to Fil(C(A))} 
Let $\mathbf{A}$ be a tense DRL-algebra and let $D\in \mathrm{tFi}(\mathbf{A})$. Then 
\[
S_{D}=D\cap C(A)
\]
is a tense filter of $C(\mathbf{A})$. 
\end{lemma}
\begin{proof}
Notice that from Corollary \ref{congruencias tDRL} and the discussion of above it is no hard to see that $w(\theta_D)=\theta_{D}\cap C(A)^{2}$. So, from Lemma \ref{congruences tICRDL} and the fact that 
\[ 
1/(\theta_{D}\cap C(A)^{2})= 1/\theta_{D}\cap C(A)=D\cap C(A)
\]
the result follows.  
\end{proof}

Let $\mathbf{A}$ be a tense DRL-algebra. Now, we proceed to describe explicitly how from a tense filter of $C(\mathbf{A})$ we can obtain a unique congruence of $\mathbf{A}$. To do so, let $S$ be a tense filter of $C(\mathbf{A})$.  Recall that from Lemma \ref{congruences tICRDL} the map 
\[
\Theta \colon \mathrm{tFi}(C(\mathbf{A}))\rightarrow  \mathsf{Con}(C(\mathbf{A}))
\]
is an isomorphism. By Proposition \ref{iso congruencias via funtor K}, the map  
\[
k \colon \mathsf{Con}(C(\mathbf{A})) \rightarrow \mathsf{Con}(K(C(\mathbf{A})))
\]
defined by $k(\theta) = \gamma_{\theta}$ is a poset isomorphism. Let $(a,b)\in K(C(A))$. Then, since $\beta_{\mathbf{A}}$ is an isomorphism, from condition (CK$^{\cdot}$) there exists a unique $z\in A$ such that $z \vee c = a$ and $\sim z \vee c = b$. Thus, it is the case that the map $r_{\mathbf{A}} \colon K(C(\mathbf{A})) \rightarrow \mathbf{A}$ defined by $r_{\mathbf{A}} (a,b) = z$ is in fact $\beta_{\mathbf{A}}^{-1}$ and therefore an isomorphism. By general reasons, the map $w_{\mathbf{A}} = r_{\mathbf{A}} \times r_{\mathbf{A}}$ induces an isomorphism between $\mathsf{Con}(K(C(\mathbf{A})))$ and $\mathsf{Con}(\mathbf{A})$. Consider then the composite $m=w_{A} k \Theta$, i.e., 
\[
m \colon \mathrm{tFi}(C(\mathbf{A})) \rightarrow \mathsf{Con}(\mathbf{A}). 
\]

\begin{lemma} \label{Fil(C(A)) to Fil(A)}
Let $\mathbf{A}$ be a tense DRL-algebra and let $S\in \mathrm{tFi}(C(\mathbf{A}))$. Then 
\[
D_{S} = \{u \in A \colon u\vee c, c\Rightarrow u\in S \}
\] 
is a tense filter of $\mathbf{A}$. 
\end{lemma}
\begin{proof}
From the discussion of above, notice that $m(S)$ is a congruence of $\mathbf{A}$. We shall prove that $m(S)=\theta_{S}$. To do so, observe first that from Theorem \ref{iso congruencias via funtor K} and Lemma \ref{congruences tICRDL} we have  
\begin{equation} \label{equation 1}
k\Theta(S) = \{((a,b),(x,y)) \in K(C(T))^{2} \colon a\circ x, b\circ y\in S\}.
\end{equation}
Thus, if we write $r_{\mathbf{A}}(a,b)=u$ and $r_{\mathbf{A}}(x,y)=v$, it is the case that $u\vee c=a$, $v\vee c=x$, $\sim u\vee c=b$ and $\sim v\vee c=y$. Since $c\leq v\vee c $, then $c\Rightarrow (v\vee c)=1$. So, we get 
\begin{displaymath}
\begin{array}{rcl}
a\Rightarrow x & = & (u\vee c)\Rightarrow (v\vee c) \\
                        & = & (u\Rightarrow (v\vee c))\wedge (c\Rightarrow (v\vee c)) \\
                        & = & (u\Rightarrow (v\vee c))\wedge 1 \\  
                        & = & u\Rightarrow (v\vee c).
\end{array}
\end{displaymath}
By the same arguments of above, we can prove also that $x \Rightarrow a = v\Rightarrow (u\vee c)$, $b\Rightarrow y=\sim u\Rightarrow (\sim v\vee c)$ and $y\Rightarrow b=\sim v\Rightarrow (\sim u\vee c)$. Then, since $S$ is an up-set, we obtain
\[
m(S)=\{(u,v)\in T^{2}\colon u\Rightarrow (v\vee c), v\Rightarrow (u\vee c), u\Rightarrow (\sim v\vee c), \sim v\Rightarrow (\sim u\vee c)\in S\}.
\]
Hence $m(S)=\theta_{S}$. We conclude by noticing that from Corollary \ref{congruencias tDRL} and the fact that $\sim u\Rightarrow c=c\Rightarrow u$ we get  $D_S=1/\theta_{S}$, as claimed.
\end{proof}

\begin{theorem}
Let $\mathbf{A}$ be a tense DRL-algebra, $D\in \mathrm{tFi}(\mathbf{A})$ and $S\in \mathrm{tFi}(C(\mathbf{A}))$. Then the assignments $D \rightarrow S_{D}$ and $S \mapsto D_{S}$ establish a mutually-inverse poset isomorphism between $\mathrm{tFi}(\mathbf{A})$ and $\mathrm{tFi}(C(\mathbf{A}))$.
\end{theorem}
\begin{proof}
We shall prove that $D_{S_{D}}=D$. On the one hand, let $u\in D$. Then from the fact that $D$ is an up-set and $u \leq c\Rightarrow u$, we get that $u\vee c$ and $c\Rightarrow u$ belong to $D\cap C(A)$. Thus, $u\in D_{S_{D}}$. On the other hand, if $u\vee c, c\Rightarrow u\in D\cap C(A)$, since $(u\vee c)\ast(c\Rightarrow u)\leq u$ and $D$ is a tense filter, we conclude $u\in D$. The proof of $S_{D_{S}}=S$ is analogue. Straightforward calculations prove that both assignments are monotone. 
\end{proof}

We conclude this section by noticing that Theorems \ref{iso congruencias via funtor K} and  \ref{congruences tICRDL}, and Corollary \ref{congruencias tDRL} may also be applied in order to make explicit the relation between the tense filters of a tense ICRDL-algebra $\mathbf{L}$ and the tense filters of $K(\mathbf{L})$. To this end, let $\Theta$ be the lattice isomorphism between $\mathrm{tFi}(\mathbf{L})$ and $\mathsf{Con}(\mathbf{L})$ of Theorem \ref{congruences tICRDL}, $w$ be the isomorphism between $\mathsf{Con}(\mathbf{L})$ and $\mathsf{Con}(K(\mathbf{L}))$ of Theorem \ref{iso congruencias via funtor K} and $\mu$ be the isomorphism between $\mathsf{Con}(K(\mathbf{L}))$ and $\mathrm{tFi}(K(\mathbf{L}))$ of Corollary \ref{congruencias tDRL}. Then, if $g$ denotes the composition of these, it is clear that $g$ establishes an isomorphism between $\mathrm{tFi}(\mathbf{L})$ and $\mathrm{tFi}(K(\mathbf{L}))$. Moreover, if $S\in \mathrm{tFi}(\mathbf{L})$ and $J\in \mathrm{tFi}(K(\mathbf{L}))$, straightforward calculations show that 
\[
g(S)=\{ (x,y)\in K(A)\colon x\in S, \neg y\in S\}=J_{S}
\]
and 
\[
g^{-1}(J)=\{ a\in L\colon (a,0)\in J\}=S_{J}.
\]
The latter discussion may be summarized in the following result.
\begin{theorem}
Let $\mathbf{L}$ be a tense ICRDL-algebra, $S\in \mathrm{tFi}(\mathbf{L})$ and $J\in \mathrm{tFi}(K(\mathbf{A}))$. Then the assignments $S\rightarrow J_{S}$ and $J\mapsto S_{J}$ establish a mutually-inverse poset isomorphism between $\mathrm{tFi}(\mathbf{L})$ and $\mathrm{tFi}(K(\mathbf{L}))$.
\end{theorem}

\section{A contextual translation}\label{sec7}

Let $\mathcal{V}$ and $\mathcal{W}$ be varieties. Let us write $\mathsf{V}$ and $\mathsf{W}$ for the categories of algebras and homomorphisms of $\mathcal{V}$ and $\mathcal{W}$, respectively. In \cite{M2018}, Moraschini proved that there is a correspondence between the nontrivial adjunctions between $\mathsf{V}$ and $\mathsf{W}$ (i.e. pairs of functors $F \colon \mathsf{V} \rightarrow \mathsf{W}$ and $G \colon \mathsf{W} \rightarrow \mathsf{V}$ such that $F$ is left adjoint to $G$ (denoted by $F\dashv G$)) and \emph{nontrivial $\kappa$-contextual translations} from the equational consequence relation relative to $\mathcal{V}$ into the equational consequence relation  relative to $\mathcal{W}$, denoted by $\models_{\mathcal{V}}$ and $\models_{\mathcal{W}}$, respectively. In this setting, $\kappa$ is assumed to be a non necessarily finite cardinal. Classical examples of $\kappa$-contextual translations are the Kolmogorov translation induced by the well known adjunction coming from the Kalman construction for distributive lattices \cite{Cig1986}, the Glivenko translation induced by Glivenko's functor that relates Boolean algebras and Heyting algebras \cite{Kol} and the G\"odel translation which relates the equational consequence relative to Heyting algebras and interior algebras \cite{G1933}.  

The aim of this section is to show that Theorem \ref{equivalence theorem} can be employed to obtain a finite nontrivial 2-contextual translation between the equational consequence relations relative to tense DRL-algebras (or tDRL-algebras, for short) and tense ICDRL-algebras (or tIRL-algebras, for short). In order to make this section self-contained, we proceed to recall the definitions that we will require on the following.

Let $X$ be a set of variables. If $\mathcal{L}$ is a propositional language we write $\mathbf{Fm}_{\mathcal{L}}(X)$ for the absolutely free algebra of $\mathcal{L}$-terms.  If $\mathcal{V}$ is a variety, we write $\mathbf{F}_{\mathcal{V}}(X)$ for the $\mathcal{V}$-free algebra over $X$. Further, if $\mathcal{L}$ is the language of $\mathcal{V}$ and $t(x_1,...,x_n)$ is an $\mathcal{L}$-formula, we also denote by $t(x_1,...,x_n)$ its image under the natural map $\mathbf{Fm}_{\mathcal{L}}(X)  \rightarrow \mathbf{F}_{\mathcal{V}}(X)$ from the term algebra $\mathbf{Fm}_{\mathcal{L}}(X)$ over $X$ onto $\mathbf{F}_{\mathcal{V}}(X)$. In particular, if $X=\{x_1,...,x_k\}$, then we write $\mathbf{F}_{\mathcal{V}}(k)$ instead of $\mathbf{F}_{\mathcal{V}}(X)$. We also denote by $\mathcal{V}$ the category associated to the variety. Let $\mathbf{A}\in \mathcal{V}$. If $S\subseteq A\times A$, we write $\mathsf{Cg}^{\mathbf{A}}(S)$ for the congruence generated by $S$. We also write $\mathsf{Cg}^{\mathbf{A}}(\vec{a},\vec{b})$, for the congruence generated by all pairs $(a_{1} , b_{1} ), \dots, (a_{N} , b_{N})$ where $\vec{a}, \vec{b}\in A^{N}$. We say that a congruence $\theta$ on $\mathbf{A}$ is \emph{finitely generated} if $\theta=\mathsf{Cg}^{\mathbf{A}}(S)$ for some finite set $S\subseteq A\times A$.  We recall that an algebra $\mathbf{A}$ in $\mathcal{V}$ is a \emph{finitely generated free algebra} if it is isomorphic to $\mathbf{F}_{\mathcal{V}}(m)$ for some finite $m$, and \emph{finitely presented} if it is isomorphic to an algebra of the form $\mathbf{F}_{\mathcal{V}}(k)/\theta$, for some $k$ finite and $\theta$ finitely generated congruence on $\mathbf{F}_{\mathcal{V}}(k)$.

If $l \colon X \rightarrow Y$ is a function, it is well known that there is a unique homomorphism $\mathbf{F}_{\mathcal{V}}(l) \colon \mathbf{F}_{\mathcal{V}}(X)\rightarrow \mathbf{F}_{\mathcal{V}}(Y)$ extending the map $\alpha(x)=l(x)$. If $(x_1,...,x_n)\in X^{n}$ and $t(x_1,...,x_n)$ is a $\mathcal{L}$-term with variables in $\{x_1,...,x_n\}$, then it is well known that $\mathbf{F}_{\mathcal{V}}(l)(t(x_1,...,x_n))=t(l(x_1),...,l(x_n))$. Let $\mathcal{F}=\{\mathbf{A}_{i}\}_{i\in I}$ be a family of algebras of $\mathcal{V}$ and let us assume w.l.o.g. that their universes are pair-wise disjoint. Let $\pi_i \colon \mathbf{F}_{\mathcal{V}}(A_{i})\rightarrow \mathbf{A}_{i}$ be the unique onto-homomorphism extending the identity of $A_{i}$. Then, if  $X=\bigcup_{i\in I} A_{i}$ and $\gamma =\mathsf{Cg}^{\mathbf{F}_{\mathcal{V}(X)}}(\bigcup_{i\in I} \text{Ker} (\pi_i))$ it is well known that $\mathbf{F}_{\mathcal{V}}(X)/\gamma$ toghether with the homomorphisms $f_{i} \colon \mathbf{A}_{i}\rightarrow \mathbf{F}_{\mathcal{V}}(X)/\gamma$ defined by $a\mapsto a/\gamma$, is isomorphic to the coproduct of $\mathcal{F}$, denoted by $\sum_{i\in I}\mathbf{A}_{i}$. In particular, if $I=\{1,...,n\}$, then we write $\mathbf{A}_1 + \ldots + \mathbf{A}_n$ instead of $\sum_{i\in I}\mathbf{A}_{i}$.

\subsection{2-contextual translations} 

Let $\mathcal{K}$ be a class of similar algebras and let $\mathcal{L}_{\mathcal{K}}$ be the language of $\mathcal{K}$. Then $\mathcal{L}_{\mathcal{K}}^{2}$ is the algebraic language whose $n$-ary operations (for every $n\in \mathbb{N}$) are all pairs of $\mathcal{L}_{\mathcal{K}}$-terms built up with the variables $X_n = \{x_j^{1}\colon 1\leq j\leq n\} \cup \{x_j^{2}\colon 1\leq j\leq n\}$. That is to say, 
\[
f((x_1^{1},x_1^{2}),...,(x_n^{1},x_n^{2}))=(t_1(x_1^{1},x_1^{2},...,x_n^{1},x_n^{2}),t_2(x_1^{1},x_1^{2},...,x_n^{1},x_n^{2})).
\]
where $t_1,t_2\in Tm_{\mathcal{L}_{\mathcal{K}}}(X_n)$. Thus, for instance, if we set the language of binary operations $\mathcal{L}_{\mathcal{K}}=\{+,\cdot\}$, then the operation $\oplus$, defined as \[\oplus ((x_1^{1},x_1^{2}),(x_2^{1},x_2^{2})):=(x_{2}^{2}\cdot (x_1^{1}+x_2^{1}), x_1^{2}\cdot x_2^{1}),\]
is a binary operation on $\mathcal{L}_{\mathcal{K}}^{2}$.

Let $\mathbf{A}\in \mathcal{K}$. We write $\mathbf{A}^{[2]}$ for the algebra of type $\mathcal{L}_{\mathcal{K}}^{2}$ with universe $A^{2}$ where every $n$-ary operation $f\in \mathcal{L}_{\mathcal{K}}^{2}$  is interpreted as
\[
f^{\mathbf{A}^{[2]}}((a_1^{1},a_1^{2}),...,(a_n^{1},a_n^{2})):=(t_1^{\mathbf{A}}(a_1^{1},a_1^{2},...,a_n^{1},a_n^{2}),t_2^{\mathbf{A}}(a_1^{1},a_1^{2},...,a_n^{1},a_n^{2})),
\]
for every $(a_j^{1},a_j^{2})\in A^{2}$, with $1\leq j\leq n$.

By the \emph{2-matrix power} of $\mathcal{K}$ me mean the class $\mathbb{I}\{\mathbf{A}^{[2]} \colon \mathbf{A}\in \mathcal{K}\}$. Such a class will be denoted by $\mathcal{K}^{[2]}$. It is well known that $\mathcal{K}^{[2]}$ is a variety if and only if $\mathcal{K}$ is a variety (Theorem 2.3 (iii) of \cite{Mc1996}). Moreover, if $\mathbf{A},\mathbf{B}\in \mathcal{K}$ and $h \colon \mathbf{A} \rightarrow \mathbf{B}$ is a homomorphism, then the map $h \times h \colon \mathbf{A}^{2}\rightarrow \mathbf{B}^{2}$ is a homomorphism. So, if we set $h^{[2]} = h\times h$, then it is easy to see that the assignments $\mathbf{A} \mapsto \mathbf{A}^{[2]}$ and $h\mapsto h^{[2]}$ establish a functor $^{[2]} \colon \mathcal{K} \rightarrow \mathcal{K}^{[2]}$.

Let $\mathcal{L}_{\mathcal{V}}$ and $\mathcal{L}_{\mathcal{W}}$ be the languages of the varieties $\mathcal{V}$ and $\mathcal{W}$, respectively. A \emph{2-translation} of $\mathcal{L}_{\mathcal{V}}$ into $\mathcal{L}_{\mathcal{W}}$ is a map $\tau \colon \mathcal{L}_{\mathcal{V}}\rightarrow \mathcal{L}_{\mathcal{W}}^{2}$ that preserves the arities of function symbols. That is to say, if $\psi$ is a $n$-ary function symbol of $\mathcal{L}_{\mathcal{V}}$, then 
\[
\tau(\psi)((x_1^{1},x_1^{2}),...,(x_n^{1},x_n^{2}))=(t_1(x_1^{1},x_1^{2},...,x_n^{1},x_n^{2}),t_2(x_1^{1},x_1^{2},...,x_n^{1},x_n^{2}))
\]
for some $t_1,t_2\in Tm_{\mathcal{L}_{\mathcal{K}}}(X_n)$. It is worth recalling that $\tau$ maps constant symbols into pairs of constant symbols. Therefore for a translation to exists, both languages are required to contain at least one constant symbol. Observe that 2-translations extend to arbitrary terms as follows: if $\lambda$ is a cardinal, we write $X_{\lambda}$ for the set $\{x_j^{1}\colon j<\lambda\}\cup \{x_j^{2}\colon j<\lambda\}$. Thus, we define recursively a map 
\[
\tau_{\ast} \colon Tm(\mathcal{L}_{\mathcal{V}}, \lambda)\rightarrow Tm(\mathcal{L}_{\mathcal{W}}, X_{\lambda})^{2}
\]
as
\begin{displaymath}
\begin{array}{rclc}
\tau_{\ast} (x_j) & :=  & (x_j^{1}, x_j^{2}), & \text{for every $j<\lambda$,}\\
\tau_{\ast}(c)& := & \tau(c). &
\end{array}
\end{displaymath}
For variables and constant symbols. If $\psi \in  \mathcal{L}_{\mathcal{V}}$ is $n$-ary and $\varphi_{1},..., \varphi_{n}\in Tm(\mathcal{L}_{\mathcal{V}}, X_{\lambda})$, then 
\begin{displaymath}
\begin{array}{rcl}
\tau_{\ast} (\psi(\varphi_{1},..., \varphi_{n})) & :=  & \tau(\psi)(\tau_{\ast}(\varphi_{1}),..., \tau_{\ast}(\varphi_{n})).
\end{array}
\end{displaymath}
Furthermore, the map $\tau_{\ast}$ may be raised to sets of equations in order to produce a new map 
\[
\tau^{\ast}:\mathcal{P}(\mathrm{Eq}(\mathcal{L}_{\mathcal{V}}, \lambda)) \rightarrow \mathcal{P}(\mathrm{Eq}(\mathcal{L}_{\mathcal{W}}, X_{\lambda})
)
\]
through the assignment
\begin{displaymath}
\begin{array}{ccc}
\Phi & \longmapsto & \{\tau_{\ast}(\varepsilon)\approx  \tau_{\ast}(\delta)\colon \varepsilon\approx \delta \in \Phi\}.
\end{array}
\end{displaymath}

A \emph{2-contextual translation} of $\models_{\mathcal{V}}$ into $\models_{\mathcal{W}}$ is a pair $(\tau, \Theta)$ where $\tau$ is a 2-translation of $\mathcal{L}_{\mathcal{V}}$ into $\mathcal{L}_{\mathcal{W}}$ and $\Theta \subseteq \mathrm{Eq}(\mathcal{L}_{\mathcal{W}},2)$ is a set of equations written with variables among $\{x_{j}\colon j<\lambda\}$ such that: 
\begin{enumerate}
\item For every $n$-ary $\psi\in \mathcal{L}_{\mathcal{V}}$: 
\[
\Theta(x_1^{1},x_1^{2})\cup \ldots \cup \Theta(x_n^{1},x_n^{2}) \models_{\mathcal{W}} \Theta(\tau_{\ast}(\psi) ((x_1^{1},x_1^{2}),...,(x_n^{1},x_n^{2}))).
\]
\item For every cardinal $\lambda$ and $\Phi \cup \{ \varepsilon \approx \delta\}\subseteq \mathrm{Eq}(\mathcal{L}_{\mathcal{V}},\lambda)$:
\begin{displaymath}
\begin{array}{ccc}
\Phi \models_{\mathcal{V}} \epsilon \approx \delta & \Rightarrow  &   \tau_{\ast}(\Phi)\cup \bigcup\limits_{j<\lambda}\Theta(x_{j}^{1},x_{j}^{2})\models_{\mathcal{W}}\tau_{\ast}(\varepsilon \approx \delta).
\end{array}
\end{displaymath}
\end{enumerate}
The set $\Theta$ is called the \emph{context} of the 2-contextual translation $(\tau, \Theta)$. We say that a 2-contextual translation is \emph{nontrivial} if there is a (nonempty) sequence $\vec{\varphi}\in Tm(\mathcal{L}_{\mathcal{W}}, \emptyset)^{2}$ of constant symbols such that if $\mathcal{W}\models \Theta(\vec{\varphi})$ then there is $i_0\in \{1,2\}$ and sequences of variables $\vec{x}=(x^1,x^2)$ and $\vec{y}=(y^1,y^2)$ such that 
\[
\Theta(\vec{x})\cup \Theta(\vec{y})\not\models_{\mathcal{W}} x^{i_0}\approx y^{i_0}.
\]
Now, we proceed to present some technical results that we will require in order to establish the main result of this section.

\begin{lemma} \label{Technical lemma adjunction}
Let $\mathbf{A}$ be a tense ICDRL-algebra. Then, for all $a,b\in A$ such that $a\cdot b=0$, there exists a unique homomorphism $\psi \colon C(\mathbf{F}_{t\mathcal{DRL}}(1))\rightarrow \mathbf{A}$ such that $\psi(z\vee c)=a$ and $\psi(\sim z\vee c)=b$.
\end{lemma}
\begin{proof}
Let $a,b\in A$ such that $a\cdot b=0$ and consider the assignment $\varepsilon(z)=(a,b)$. If we take $h \colon \mathbf{F}_{t\mathcal{DRL}}(1)\rightarrow K(\mathbf{A})$ as the unique homomorphism that extends $\varepsilon$, since $C\dashv K$ (by Theorem \ref{equivalence theorem}), there exists a unique homomorphism $\psi \colon C(\mathbf{F}_{t\mathcal{DRL}}(1))\rightarrow \mathbf{A}$. For general reasons, it follows that $\psi=\alpha^{-1}_{\mathbf{A}}C(h)$, where $\alpha_{\mathbf{A}}$ is the map of Lemma \ref{lemma alpha map}. Observe that 
\begin{displaymath}
\begin{array}{ccl}
\psi(\sim z\vee c)  &  = & \alpha^{-1}_{\mathbf{A}}(h(\sim z\vee c)) \\
                             & = & \alpha^{-1}_{\mathbf{A}}(\sim h(z)\vee (0,0))) \\
                             & = & \alpha^{-1}_{\mathbf{A}}(\sim (a,b)\vee (0,0)) \\
                             & = & \alpha^{-1}_{\mathbf{A}}((b,a)\vee (0,0)) \\
                             & = & \alpha^{-1}_{\mathbf{A}}(b,0) \\
                             & = & b. \\
\end{array}   
\end{displaymath}
In a similar fashion it can be proved that $\psi(z\vee c)=a$. Now, suppose that there exists another homomorphism $\psi' \colon C(\mathbf{F}_{t\mathcal{DRL}}(1))\rightarrow \mathbf{A}$ such that $\psi'(z\vee c)=a$ and $\psi'(\sim z\vee c)=b$. Then, since $C\dashv K$ (Theorem \ref{equivalence theorem}), there exist a unique $\psi'_{\ast} \colon \mathbf{F}_{t\mathcal{DRL}}(1)\rightarrow K(\mathbf{A})$. We stress that from general reasons, $\psi'_{\ast}=K(\psi')\beta_{\mathbf{F}_{t\mathcal{DRL}}(1)}$. Therefore,
\begin{displaymath}
\begin{array}{ccl}
\psi'_{\ast}(z) & = & K(\psi')(z\vee c,\sim z\vee c) \\
                     & = & (\psi'(z\vee c),\psi'(\sim z\vee c)) \\
                     & = & (a,b) \\
                     & = & \varepsilon(z) \\
\end{array}
\end{displaymath}
so $\psi'_{\ast}=h$ and consequently, $\psi'=\psi$, as claimed.
\end{proof}

\begin{corollary} \label{coro: ampliacion}
Let $\mathbf{A}$ be a tense ICDRL-algebra. Then, for all $a_1, a_2, b_1,b_2\in A$ such that $a_1\cdot b_1=0$ and $a_2\cdot b_2=0$, there exists a unique homomorphism $\mu \colon C(\mathbf{F}_{t\mathcal{DRL}}(2))\rightarrow \mathbf{A}$ such that $\mu(x\vee c)=a_1$, $\mu(y\vee c)=a_2$, $\mu(\sim x\vee c)=b_1$ and $\mu(\sim y\vee c)=b_2$.   
\end{corollary}
\begin{proof}
Consider the assignment defined by $\varepsilon(x)=(a_1,b_1)$ and $\varepsilon(y)=(a_2,b_2)$, and take $h \colon \mathbf{F}_{t\mathcal{DRL}}(2)\rightarrow K(\mathbf{A})$ as the unique homomorphism extending $\varepsilon$. It is no hard to see that the rest of the proof is analogue to the one of Lemma \ref{Technical lemma adjunction}.
\end{proof}

Let $\mathcal{V}$ be a variety, $\mathbf{A}\in \mathcal{V}$ and $\vec{a}, \vec{b}\in A^{N}$, with $N\in \mathbb{N}$. It is well known that the canonical homomorphism $\nu \colon \mathbf{A}\rightarrow \mathbf{A}/\mathsf{Cg}^{\mathbf{A}}(\vec{a}, \vec{b})$ has the universal property of identify $\vec{a}$ with $ \vec{b}$, in the sense that for every homomorphism $h \colon \mathbf{A} \rightarrow \mathbf{B}$ such that $g(a_i)=g(b_i)$ for every $1\leq i \leq N$, there exists a unique homomorphism $g \colon \mathbf{A}/\mathsf{Cg}^{\mathbf{A}}(\vec{a}, \vec{b})\rightarrow \mathbf{B}$ such that $h\nu=g$. This fact implies that in order to prove that an algebra $\mathbf{C}$ of $\mathcal{V}$ is isomorphic to $\mathbf{A}/\mathsf{Cg}^{\mathbf{A}}(\vec{a}, \vec{b})$ it is enough to find a homomorphism with domain $\mathbf{A}$ with the universal property mentioned above. This observation will be crucial for showing the following result.

\begin{theorem}\label{theorem C}
 $C(\mathbf{F}_{t\mathcal{DRL}}(1))$ is isomorphic to $\mathbf{F}_{t\mathcal{ICRL}}(2)/\mathsf{Cg}^{\mathbf{F}_{t\mathcal{IRL}}(2)}(x\cdot y, 0)$.
\end{theorem}
\begin{proof}
In order to prove our claim, regard the assignment $\alpha(x)=z\vee c$ and $\alpha(y)=\sim z\vee c$ and let $h \colon \mathbf{B} \rightarrow C(\mathbf{F}_{t\mathcal{DRL}}(1))$ be the homomorphism that extends $\alpha$. Notice that for every $t(x,y)\in \mathbf{B}$, it is the case that $h(t(x,y))=t^{C(\mathbf{F}_{t\mathcal{DRL}}(1))}(z\vee c, \sim z\vee c)$. In particular we have that  $h(x\cdot y)= c = 0^{C(\mathbf{F}_{t\mathcal{DRL}}(1))}$. In order to prove our claim, we will show that $h$ is universal with respect to the homomorphisms $f \colon \mathbf{B} \rightarrow \mathbf{A}$ such that $f(x\cdot y)=0^{\mathbf{A}}$. To do so, if we consider such an $f$, then $f(x)\ast^{\mathbf{A}} f(y)=0^{\mathbf{A}}$, thus from Lemma \ref{Technical lemma adjunction}, there exists a unique homomorphism $\psi \colon C(\mathbf{F}_{t\mathcal{DRL}}(1))\rightarrow \mathbf{A}$ such that $\psi(z\vee c)=f(x)$ and $\psi(\sim z\vee c)=f(y)$. We will prove that the diagram below
\begin{displaymath}
\xymatrix{
\mathbf{B} \ar[r]^-{h} \ar[dr]_-{f} & C(\mathbf{F}_{t\mathcal{DRL}}(1)) \ar[d]^-{\psi} \\
& \mathbf{A}
}
\end{displaymath}
commutes. Indeed, if $t(x,y)\in \mathbf{B}$, then we have 
\begin{displaymath}
\begin{array}{ccl}
\psi h(t(x,y)) & = & \psi(t^{C(\mathbf{F}_{t\mathcal{DRL}}(1))}(z\vee c, \sim z\vee c)) \\
                   & = & t^{\mathbf{A}}(\psi(z\vee c), \psi(\sim z\vee c)) \\
                   & = &  t^{\mathbf{A}}(f(x),f(y)) \\
                   & = & f(t(x,y)). \\
\end{array}
\end{displaymath}
This concludes the proof.
\end{proof}

\begin{corollary} \label{coro: description free 2 x 2}
Let $X_2$ be the set $\{x^{1}_1,x^{1}_2, x^{2}_1, x^{2}_2\}$ and let $\mathbf{C}=\mathbf{F}_{t\mathcal{IRL}}(X_2)$. Then, if $\theta= \mathsf{Cg}^{\mathbf{C}}(x^{1}_1\cdot x^{2}_1, 0)\vee \mathsf{Cg}^{\mathbf{C}}(x^{1}_2\cdot x^{2}_2, 0)$, the algebras $C(\mathbf{F}_{t\mathcal{DRL}}(2))$ and $\mathbf{C}/\theta$ are isomorphic. 
\end{corollary}
\begin{proof}
Let $\mathbf{B}=\mathbf{F}_{t\mathcal{ICRL}}(2)$ and $\delta= \mathsf{Cg}^{\mathbf{F}_{t\mathcal{IRL}}(2)}(x\cdot y, 0)$. Then, from Theorem \ref{theorem C}, the fact that $C$ and $\mathbf{F}_{t\mathcal{DRL}}$ preserve coproducts, and the well known description on coproducts of algebras, the following holds:
\begin{displaymath}
\begin{array}{rcl}
C(\mathbf{F}_{t\mathcal{DRL}}(2)) & \cong &  C(\mathbf{F}_{t\mathcal{DRL}}(1)) + C(\mathbf{F}_{t\mathcal{DRL}}(1)) \\
                                                       & \cong & \mathbf{B}/\delta + \mathbf{B}/\delta  \\
                                                       & \cong & \mathbf{C}/\mathsf{Cg}^{\mathbf{C}}((x^{1}_1\cdot x^{2}_1, 0),(x^{1}_2\cdot x^{2}_2, 0))  \\
                                                       & \cong & \mathbf{C}/\theta, 
\end{array}
\end{displaymath}
i.e., $C(\mathbf{F}_{t\mathcal{DRL}}(2))$ and $\mathbf{C}/\theta$ are isomorphic.
\end{proof}

Next, we describe the procedure to obtain the contextual 2-translation associated to the adjunction $C\dashv K$. If $\psi$ is a $n$-ary  function symbol on the language of tense DRL-algebras, by Theorem 4.3 of \cite{M2016}, the 2-translation $\tau$ is determined by a homomorphism $\tau(\psi)$ that makes the following diagram commutes: 
\begin{displaymath}
\xymatrix{
& & \ar@{-->}[ddll]_-{\tau(\psi)} \mathbf{F}_{tICRL}(2) \ar[d]^-{h} \\
& & C(\mathbf{F}_{tDRL}(1)) \ar[d]^-{C(\psi)} \\
\mathbf{F}_{tICRL}(X_n) \ar[rr]_-{\pi_n} & & C(\mathbf{F}_{tDRL}(n))
}
\end{displaymath}
We stress that the existence of $\tau(\psi)$ is granted because $\mathbf{F}_{tICRL}(2)$ is onto-projective in $\mathsf{tICRL}$. Furthermore, the map $\tau \colon \mathcal{L}_{tDRL}\rightarrow \mathcal{L}_{tICRL}^{2}$ can be identified with its values on the generators. I.e. $\tau(\psi) := (\tau(\psi)(x_{1}^{1}),\tau(\psi)(x_{1}^{2}))$. Notice that in order to find such a map, it is enough to find terms $t(x_{1}^{1},x_{1}^{2},\ldots, x_{n}^{1},x_{n}^{2})$ and $s(x_{1}^{1},x_{1}^{2},\ldots, x_{n}^{1},x_{n}^{2})$ in $\mathbf{F}_{tICRL}(X_n)$ such that $C(\psi)(h(x))=h(t(x_{1}^{1},x_{1}^{2},\ldots, x_{n}^{1},x_{n}^{2}))$ and $C(\psi)(h(y))=h(s(x_{1}^{1},x_{1}^{2},\ldots, x_{n}^{1},x_{n}^{2}))$. In the next result we will see that in our case, such terms can be obtained from the operations we defined in Section \ref{sec5} at the moment of establishing the functor $K$. 

\begin{lemma}\label{lem: 2-translation}
The adjunction $C\dashv K$ induces the 2-translation  $\tau \colon \mathcal{L}_{tDRL}\rightarrow \mathcal{L}_{tICRL}^{2}$ defined as follows:
\begin{displaymath}
\begin{array}{rcl}
\tau(\wedge)((x_{1}^{1},x_{1}^{2}), (x_{2}^{1},x_{2}^{2})) & := & (x_{1}^{1}\wedge x_{2}^{1},x_{1}^{2}\vee x_{2}^{2})\\
\tau(\vee)((x_{1}^{1},x_{1}^{2}), (x_{2}^{1},x_{2}^{2})) & := & (x_{1}^{1}\vee x_{2}^{1},x_{1}^{2}\wedge x_{2}^{2})\\
\tau(\ast)((x_{1}^{1},x_{1}^{2}), (x_{2}^{1},x_{2}^{2})) & := & (x_{1}^{1}\cdot x_{2}^{1}, (x_{1}^{1}\rightarrow x_{2}^{2})\wedge (x_{2}^{1}\rightarrow x_{1}^{2}))\\
\tau(\Rightarrow)((x_{1}^{1},x_{1}^{2}), (x_{2}^{1},x_{2}^{2})) & := & ((x_{1}^{1}\rightarrow x_{2}^{1})\wedge (x_{2}^{2}\rightarrow x_{1}^{2}), x_{1}^{1}\cdot x_{2}^{2})\\
\tau(G)(x_{1}^{1},x_{1}^{2}) & := & (Gx_{1}^{1},Fx_{1}^{2})\\
\tau(F)(x_{1}^{1},x_{1}^{2})& := & (Fx_{1}^{1},Gx_{1}^{2})\\
\tau(H)(x_{1}^{1},x_{1}^{2}) & := & (Hx_{1}^{1},Px_{1}^{2})\\
\tau(P)(x_{1}^{1},x_{1}^{2}) & := & (Px_{1}^{1},Hx_{1}^{2})\\
\tau(\sim)(x_{1}^{1},x_{1}^{2}) & := & (x_{1}^{2},x_{1}^{1})\\
\tau(c) & := & (0,0)\\
\tau(0) & := & (0,1)\\
\tau(1) & := & (1,0).\\
\end{array}
\end{displaymath}    
\end{lemma}
\begin{proof}  
We start by recalling that $C(G) \colon C(\mathbf{F}_{tDRL}(1))\rightarrow C(\mathbf{F}_{tDRL}(1))$ maps a term $u(z)$ into $u(G(z))$ and by Theorem \ref{theorem C}, $h \colon \mathbf{F}_{tICRL}(2)\rightarrow C(\mathbf{F}_{tDRL}(1))$ maps a term $t(x_{1}^{1},x_{1}^{2})$ into $t^{C(\mathbf{F}_{t\mathcal{DRL}}(1))}(z\vee c, \sim z\vee c)$. Now, let us consider $t(x_{1}^{1},x_{1}^{2})=G(x_{1}^{1})$ and $s(x_{1}^{1},x_{1}^{2})=F(x_{1}^{2})$. Observe that from Proposition \ref{propiedades tDRLc}, it is the case that $h(G(x_{1}^{1}))=C(G)(h(x_{1}^{1}))$. On the other hand, from Definition \ref{def: tDRL} (t1) and (t2), we have 
\begin{displaymath}
\begin{array}{rcl}
h(F(x_{1}^{2})) & = & F(\sim z\vee c) \\
                        & = & \sim G(z\wedge c) \\
                        & = & \sim (G(z)\wedge c) \\
                        & = & \sim G(z)\vee c \\
                        & = & C(G)(\sim z\vee c) \\
                        & = & C(G)(h(x_{1}^{2})). \\
\end{array}
\end{displaymath}
Hence, we can set $\tau(G)(x_{1}^{1},x_{1}^{2}) = (Gx_{1}^{1},Fx_{1}^{2})$ as claimed. The proof of the statements about $\tau(H)$, $\tau(F)$,  $\tau(P)$ and $\tau(\sim)$ is similar.  Now, in order to show that 
\[
\tau(\ast)((x_{1}^{1},x_{1}^{2}), (x_{2}^{1},x_{2}^{2})) = (x_{1}^{1}\cdot x_{2}^{1}, (x_{1}^{1}\rightarrow x_{2}^{2})\wedge (x_{2}^{1}\rightarrow x_{1}^{2})),
\] 
we will proceed as next we will describe. Let us consider the following diagram 
\begin{displaymath}
\xymatrix{
& & \mathbf{F}_{t\mathcal{IRL}}(2) \ar[d]^-{h} \ar@/_/@{-->}[ddll]_-{\tau(\ast)}\\
& & C(\mathbf{F}_{t\mathcal{DRL}}(1)) \ar[d]^-{C(\ast)}\\
\mathbf{F}_{t\mathcal{IRL}}(X_2) \ar[r]^-{\varphi_{\theta}} \ar@/_1.5pc/[rr]_-{\pi_2} & \mathbf{F}_{t\mathcal{IRL}}(X_2)/\theta \ar[r]^-{\lambda}  & C(\mathbf{F}_{t\mathcal{DRL}}(2)) 
}  
\end{displaymath}
in which $\theta$ denotes  $\mathsf{Cg}^{\mathbf{F}_{tICRL}(X_2)}(x_1\cdot y_1,0)\vee \mathsf{Cg}^{\mathbf{F}_{tICRL}(X_2)}(x_2\cdot y_2,0)$, $\varphi_{\theta}$ is its respective quotient homomorphism and $\lambda$ is the isomorphism of  Corollary \ref{coro: description free 2 x 2}. Having into account that $\pi_2=\lambda \varphi_{\theta}$, it is clear that proving $\pi_2\tau(\ast)=C(\ast)h$ is equivalent to proof that $\lambda \varphi_{\theta}\tau(\ast)=C(\ast)h$. Thus, since $C(\ast) \colon C(\mathbf{F}_{tDRL}(1))\rightarrow C(\mathbf{F}_{tDRL}(2))$ maps a term $u(z)$ into $u(x\ast y)$, if $s(x^{1}_1,x^{2}_1,x^{1}_2,x^{2}_2)$ and $t(x^{1}_1,x^{2}_1,x^{1}_2,x^{2}_2)$ are terms of $F_{tICRL}(X_2)$ such that $\pi_2(s(x^{1}_1,x^{2}_1,x^{1}_2,x^{2}_2))=C(\ast)(h(x))$ and $\pi_2(t(x^{1}_1,x^{2}_1,x^{1}_2,x^{2}_2))=C(\ast)(h(y))$ the latter equation leads us to prove the following 
\begin{displaymath}
\begin{array}{c}
s^{C(\mathbf{F}_{tDRL}(2))}(\lambda(x^{1}_1/\theta),\lambda(x^{2}_1/\theta),\lambda(x^{1}_2/\theta),\lambda(x^{2}_2/\theta))= (x\ast y)\vee c, \\
\\
t^{C(\mathbf{F}_{tDRL}(2))}(\lambda(x^{1}_1/\theta),\lambda(x^{2}_1/\theta),\lambda(x^{1}_2/\theta),\lambda(x^{2}_2/\theta))= \sim (x\ast y)\vee c.
\end{array}   
\end{displaymath}
Now, since $x^{1}_1/\theta\cdot x^{2}_1/\theta=0/\theta$ and $x^{1}_2/\theta\cdot x^{2}_2/\theta=0/\theta$, from Corollary \ref{coro: ampliacion}, we can assume w.l.o.g. that $\lambda(x^{1}_1/\theta)=x\vee c$, $\lambda(x^{1}_2/\theta)=y\vee c$, $\lambda(x^{2}_1/\theta)=\sim x\vee c$ and $\lambda(x^{2}_2/\theta)=\sim y\vee c$. Therefore, by taking  $s(x^{1}_1,x^{2}_1,x^{1}_2,x^{2}_2)=x^{1}_1\cdot x^{1}_2$,  $t(x^{1}_1,x^{2}_1,x^{1}_2,x^{2}_2)=(x_{1}^{1}\rightarrow x_{2}^{2})\wedge (x_{2}^{1}\rightarrow x_{1}^{2})$ and having in mind that $\lambda$ is a homomorphism, in order to prove our claim, we shall prove the following
\begin{displaymath}
\begin{array}{rcl}
\lambda(x^{1}_1/\theta)\cdot \lambda(x^{1}_2/\theta)& = &(x\ast y)\vee c,  \\
\end{array}
\end{displaymath}
\begin{displaymath}
\begin{array}{rcl}
(\lambda(x_{1}^{1}/\theta)\Rightarrow \lambda(x_{2}^{2}/\theta)) \wedge  (\lambda(x_{2}^{1}/\theta) \Rightarrow \lambda(x_{1}^{2}/\theta)) & = & \sim (x\ast y)\vee c.
\end{array}
\end{displaymath}
But notice that these are precisely the equations (1) and (2) of Lemma \ref{lem: technical DRL}. So 
\[
\tau(\ast)((x_{1}^{1},x_{1}^{2}), (x_{2}^{1},x_{2}^{2})) =  (x_{1}^{1}\cdot x_{2}^{1}, (x_{1}^{1}\rightarrow x_{2}^{2})\wedge (x_{2}^{1}\rightarrow x_{1}^{2})),
\]
as desired. The proof of the statements about $\tau(\Rightarrow)$, $\tau(\vee)$ and $\tau(\wedge)$ are analogue. The details on the proof of $\tau(c)$, $\tau(0)$ and $\tau(1)$ are left to the reader. 
\end{proof}

In addition to Lemma \ref{lem: 2-translation}, we stress that from Theorem \ref{theorem C}, we may identify the generating set $\{(x\cdot y,0)\}$ with the set equations $\Theta=\{x\cdot y\approx 0\}$. Then, from Theorem 4.3 of \cite{M2018}, it turns out that the pair $(\tau, \Theta)$ is a 2-contextual translation 
of $\models_{tDRL}$ into $\models_{tICRL}$. Further is true. Since $C$ is faithful (Theorem \ref{equivalence theorem}), then by Lemma 6.4 of \cite{M2016} we are able to conclude:

\begin{theorem}
For every cardinal $\lambda$ and $\Phi \cup \{ \varepsilon \approx \delta\}\subseteq \mathrm{Eq}(\mathcal{L}_{tDRL},\lambda)$, 
\[
\Phi \models_{tDRL} \varepsilon \approx \delta \Longleftrightarrow \tau_{\ast}(\Phi)\cup \{(x_j^{1}\cdot x_j^{2},x_j^{1}\cdot x_j^{2})\approx (0,0)\colon j<\lambda \}\models_{tICRL}\tau_{\ast}(\varepsilon \approx \delta).
\]
where $(\tau, \Theta)$ is the contextual translation of $\models_{tDRL}$ into $\models_{tICRL}$ induced by $C$.
\end{theorem}

\end{document}